\documentclass{gtart}

\usepackage{amsmath, amssymb, latexsym, amscd, graphicx, epstopdf}
\usepackage{euscript}
\usepackage{subfigure}
\usepackage{epsfig}
\usepackage{psfrag}

\usepackage[inner=30mm, outer=30mm, textheight=225mm]{geometry}

\def\bkC{{\rm \kern.24em \vrule width.05em height1.4ex depth-.05ex 
\kern-.26em C}}
\def\C{\bkC}
\def\bksC{{\rm \kern.24em \vrule width.05em height1ex depth-.05ex 
\kern-.26em C}}

\def\bkE{{\rm I\kern-.22em E}}
 
\def\bkH{{\rm I\kern-.22em H}}
\def\H{\bkH} 
\def\bkN{{\rm I\kern-.17em N}}

\def\bkQ{{\rm \kern.24em \vrule width.05em height1.4ex depth-.05ex 
\kern-.26em Q}}

\def\bkR{{\rm I\kern-.17em R}}
\def\RR{\bkR}
\def\bkZ{{\rm Z\kern-.32em Z}}
\def\Z{\bkZ}
\def\bksZ{{\rm Z\kern-.22em Z}}

\def\PSL{PSL_2(\C)}

\def\D{\mathfrak{D}}

\def\tri{\mathcal{T}}

\def\n{\mathcal{N}}
\def\l{\mathcal{L}}
\def\m{\mathcal{M}}

\DeclareMathOperator{\lk}{lk}
\DeclareMathOperator{\ima}{Im}
\DeclareMathOperator{\Ker}{Ker}
\DeclareMathOperator{\im}{im}
\DeclareMathOperator{\h}{h}

\DeclareMathOperator{\vol}{vol}

\DeclareMathOperator{\Hess}{Hess}
\DeclareMathOperator{\Span}{Span}
\DeclareMathOperator{\rank}{rank}
\DeclareMathOperator{\ind}{\eta}

\DeclareMathOperator{\Lk}{Lk}
\DeclareMathOperator{\Homo}{H}
\DeclareMathOperator{\diff}{d}

\DeclareMathOperator{\TAS}{TAS}
\DeclareMathOperator{\STAS}{TAS_\l}

\newcommand\restr[2]{{
  \left.\kern-\nulldelimiterspace 
  #1 
  \vphantom{\big|} 
  \right|_{#2} 
  }}

\theoremstyle{plain}

\newtheorem{thm}{Theorem}
\newtheorem*{thm*}{Theorem}
\newtheorem{lem}[thm]{Lemma}
\newtheorem*{lem*}{Lemma}
\newtheorem{cor}[thm]{Corollary}
\newtheorem*{cor*}{Corollary}

\newtheorem*{cla*}{Claim}

\newtheorem*{pro*}{Proposition}

\newtheorem*{rem*}{Remark}

\newtheorem*{defn*}{Definition}

\addtocounter{MaxMatrixCols}{4}

\begin{document}
\title{Pseudo-developing maps for ideal triangulations II:\\
Positively oriented ideal triangulations of cone-manifolds} 
\author{Alex Casella, Feng Luo and Stephan Tillmann}

\begin{abstract}
We generalise work of Young-Eun Choi to the setting of ideal triangulations with vertex links of arbitrary genus, showing that the set of all (possibly incomplete) hyperbolic cone-manifold structures  realised by positively oriented hyperbolic ideal tetrahedra on a given topological ideal triangulation and with prescribed cone angles at all edges is (if non-empty) a smooth complex manifold of dimension the sum of the genera of the vertex links. Moreover, we show that the complex lengths of a collection of peripheral elements give a local holomorphic parameterisation of this manifold.
\end{abstract}


\maketitle


\section{Introduction}

The complement $N$ of the vertices in a triangulated orientable 3--dimensional pseudo-manifold $P$ carries a complete hyperbolic cone-manifold structure, where the singular locus is contained in the 1--skeleton and each ideal tetrahedron develops into an ideal hyperbolic tetrahedron. For instance, such a structure is obtained by realising each ideal tetrahedron in $N$ as a regular ideal hyperbolic 3--simplex (see \cite{st} and \S\ref{sec:Cone-deformation variety}). Let $\tri$ denote the (topological) ideal triangulation of $N$ and $E$ the set of ideal edges. For any prescribed cone angles $\kappa\co E \to \RR,$ let $\D^+(\tri, \kappa)$ be the set of all (possibly incomplete) hyperbolic cone-manifold structures with the ideal tetrahedra in $\tri$ realised by positively oriented hyperbolic tetrahedra and with the prescribed cone angles $\kappa.$ We show that if this set is non-empty, then it is a smooth complex manifold of dimension the sum of the genera of the vertex links (Corollary~\ref{for:level set is smooth}). This generalises the first main theorem of Choi~\cite{ch}. This result is deduced as a consequence of a more general result (Theorem~\ref{thm_dimension}), which is essentially due to Neumann~\cite{ne}. Moreover, we show that the complex lengths of a collection of non-trivial peripheral elements, $g$ for each vertex link of genus $g$, give a local holomorphic parametrisation of $\D^+(\tri, \kappa)$ (Corollary~\ref{cor:parametrication}). This is achieved through a generalisation (Theorem~\ref{injective_thm}) of the second main theorem of Choi~\cite{ch}.

Our approach is different from Choi's and includes new results on the interplay between tangential angle structures, the boundary map and the study of the volume function. It builds on previous work of Neumann~\cite{ne} as well as Futer and Gu\'eritaud \cite{FuGu}.

The results of this paper show that the examples of pseudo-manifolds $M$ with spherical vertex links in \cite{st} exhibit generic behaviour---for any prescribed cone angles at the edges, the manifold $\D^+(M, \kappa)$ is either 0--dimensional or empty. We conclude this paper by giving two examples, which both have \emph{two components where one might have expected one}. The first example is a once-cusped hyperbolic 3--manifold of finite volume with the property that the two discrete and faithful characters lie on different components of the $\PSL$--character variety. It was found via an application of the volume function explained to us by Nathan Dunfield.
The second example is a manifold with two boundary components---a torus and a genus two surface. In this example, one can find distinct prescribed cone angles $\kappa_1$ and $\kappa_2$ for the edges of an ideal triangulation with the property that $\kappa_1$ and $\kappa_2$ differ by $2\pi$ at two edges. In contrast, for manifolds with only torus boundary components, a standard Euler characteristic argument shows that at each edge the total cone angle has to be exactly $2\pi.$


\section{Preliminaries}


\subsection{Conventions for vector spaces}

If $X$ is a finite set, $\RR^X$ denotes the real vector space of all functions $X\to \RR$ and we assume that $\RR^X$ has the standard inner product, so that 
$$X^\star = \{ x^\star \in \RR^X \mid x\in X \text{ and for all }  y \in X: x^\star(y) =  \delta_{xy}\}$$
is an orthonormal basis of $\RR^X,$ where
Kronecker's notation $\delta_{xy} = \begin{cases} 1 & \text{if } y=x\\ 0 & \text{if } y\neq x \end{cases}$  is used. 


\subsection{Pseudo-manifolds and triangulations}
\label{subsec:pseudo-manifolds}

Let $\widetilde{\Delta}$ be a finite union of pairwise disjoint, oriented Euclidean $3$--simplices, and $\Phi$ be a family of orientation-reversing affine isomorphisms pairing the facets in $\widetilde{\Delta},$ with the properties that $\varphi \in \Phi$ if and only if $\varphi^{-1}\in \Phi,$ and every codimension-one facet is the domain of a unique element of $\Phi.$ The elements of $\Phi$ are termed \emph{face pairings}. The quotient space $P = \widetilde{\Delta}/\Phi$ with the quotient topology is a closed, orientable $3$--dimensional pseudo-manifold, and the quotient map is denoted $p\co \widetilde{\Delta} \to P.$ The triple $\tri = ( \widetilde{\Delta}, \Phi, p)$ is a \emph{(singular) triangulation} of $P.$ The adjective \emph{singular} is usually omitted, and we will not need to distinguish between the cases of a simplicial or a singular triangulation.
We will always assume that $P$ is connected. In the case where $P$ is not connected, the results of this paper apply to its connected components.

We will use the following notation:
$$T = \{\sigma_i \} = \widetilde{\Delta}^{(3)},\qquad\qquad
E = \{e_j\}=P^{(1)}, \qquad\qquad
V = \{v_k\}=P^{(0)}.
$$
Note that $E$ and $V$ are equivalence classes of 1--simplices and 0--simplices of $\widetilde{\Delta}$ respectively.

The set of non-manifold points of $P$ is contained in the $0$--skeleton. Denote this set $V_s \subseteq V=P^{(0)}.$ The cases of interest are usually when $V_s = \emptyset$ or $V_s = V$. In the first case $P$ is a closed 3--manifold. In the second case $\tri$ restricts to an \emph{ideal triangulation} of the topologically finite, non-compact 3--manifold $N = P \setminus P^{(0)}$ and $P$ is the \emph{end-compactification} of $N.$

For each vertex $v \in V$, $\Lk(v)$ is a closed orientable surface of genus $g_v\geq0$, with a triangulation $\tri_v$ induced from $\tri$. We will repeatedly make use of the following fact, which follows from a direct Euler characteristic calculation:

\begin{lem} \label{dimension_count}
$\displaystyle |T| - |E| + |V| = \sum_{v \in V} g_v$.
\end{lem}


\subsection{Quadrilateral index}

Let $\sigma$ be a 3--simplex. A \emph{quadrilateral type} $q$ in $\sigma$ is a partition of its set of vertices into two sets of cardinality two. The name alludes to its geometric realisation as a properly embedded quadrilateral disc separating a pair of opposite 1--simplices $e_0$ and $e_1$. See \cite{ti} for an exposition of this well-known geometric viewpoint, which goes back to Haken~\cite{H}. It will be convenient to regard a quadrilateral type as a set containing these two opposite 1--simplices, and we write
$$ \{ e_0, e_1\} = q  < \sigma$$
and say that $e_0$ and $e_1$ \emph{face} $q.$

There are precisely three quadrilateral types in $\sigma.$ There is a natural action of the symmetric group $\text{Sym}(4)$ on the set of vertices of $\sigma$. Choose an orientation of $\sigma.$ The alternating group $\text{Alt(4)}$ fixes the orientation and permutes the three quadrilateral types. The stabiliser of a quadrilateral type is the Klein four group $K.$ So there is a natural faithful action of the cyclic group $C_3 \cong \text{Alt(4)}/K$ on the set of all quadrilateral types, and hence a natural cyclic order on that set. Throughout, $q, q', q''$ denote the three quadrilateral types in the 3--simplex $\sigma$ and the action of $C_3$ is indicated by the prime mark, so $(q')' = q''$ and $(q'')' = q$ and the natural cyclic order is therefore given by $q \rightarrow q' \rightarrow q'' \rightarrow q.$ 

Denote $\Box$ the set of all normal quadrilateral types in $\widetilde{\Delta}.$ 
The quotient map $p\co \widetilde{\Delta} \to P$ induces a natural map $p\co \widetilde{\Delta}^{(1)} \to P^{(1)}=E.$
For any $e \in E$ and $q \in \Box,$ the number of edges in the preimage  $p^{-1}(e) \subset \widetilde{\Delta}^{(1)}$ facing $q$ is:
$$i(q,e) =|q \cap p^{-1}(e)| \in \{0, 1, 2\}.$$ If $i(q,e)>0,$ we say that $e$ faces $q,$ and write $q \sim e.$


\subsection{Cone-deformation variety}
\label{sec:Cone-deformation variety}

The \emph{cone-deformation variety} $\D (\tri; \star)$ is the set of all $(z, \xi) \in \C^{\Box}\times (S^1)^E$ satisfying:
\begin{itemize}
\item[$i)$] for each edge $e \in E$,
\begin{equation} \label{glueing_eq}
\prod_{q \in \Box} z(q)^{i(q,e)}=\xi(e),
\end{equation}
\item[$ii)$] for each $q \in \Box$,
\begin{equation} \label{param_eq}
z(q') = \frac{1}{1-z(q)}.
\end{equation}
\end{itemize}
Equation $(\ref{param_eq})$ is the \emph{parameter relation} for $q$. Applying the cyclic ordering gives:
$$z(q'') = 1 - \frac{1}{z(q)} = \frac{z(q) - 1}{z(q)} \qquad\text{and}\qquad z(q)z(q')z(q'')=-1.$$
Equation $(\ref{glueing_eq})$ is the \emph{cone-hyperbolic gluing equation} for $e.$ Multiplying all of these equations gives the identity
$$\prod_{e\in E} \xi(e) = 1.$$
If $\xi(e)=1,$ then equation $(\ref{glueing_eq})$ is the usual \emph{hyperbolic gluing equation} of $e.$ 

We have the projections onto the factors 
$$s\co \D (\tri; \star) \to \C^{\Box} \qquad \text{and} \qquad  c \co \D (\tri; \star)\to (S^1)^E,$$
where $s$ gives the \emph{shapes} of the tetrahedra and $c$ gives the \emph{curvature} at the edges.

Denote the upper half plane in $\C$ by $\H$ and for each $\xi \in (S^1)^E$ define 
$$\D^+(\tri; \xi) = \D (\tri; \star) \; \cap \; ( \H^{\Box}\times\{\xi\}).$$ 
It was observed in \cite{st} that the cone-deformation variety is non-empty for any triangulation since one may choose each $z(q)=\frac{1}{2}(1+\sqrt{-3})$ to be the shape of the regular hyperbolic ideal  3--simplex, and this turns $N$ into a complete hyperbolic cone-manifold. Hence there always exists $\xi \in (S^1)^E$ such that $\D^+(\tri; \xi)\neq \emptyset.$

One can study the (topological) connected components of $\D^+(\tri; \xi)$ via the set of all $\kappa\in \RR^E$ with the property that $z\in  \H^{\Box}$ satisfies the parameter relations and for each edge $e,$ we have:
\begin{equation} \label{glueing_eq_logs}
\sum_{q \in \Box} {i(q,e)} \; \log(z(q))=\kappa(e),
\end{equation}
and
\begin{equation}
\xi(e) = \exp(i\kappa(e)).
\end{equation}
Throughout this paper, $\log$ is the standard branch on $\C \setminus ( - \infty,0 ]$ unless stated otherwise. It follows from analytic continuation that on each connected component of $\D^+(\tri; \xi),$ the left-hand side of (\ref{glueing_eq_logs}) is constant.

Below Corollaries~\ref{for:level set is smooth} and \ref{cor:parametrication} imply that if $\D^+(\tri; \xi)$ is non-empty, then it is 
a smooth complex manifold of dimension the sum of the genera of the vertex links, and each of its components has a holomorphic parametrisation by the holonomies of peripheral elements, one for each genus. This will be proved using the more general \emph{complex-curvature} and \emph{log-curvature} maps defined in \S\ref{sec:log-curvature map}, which respectively generalise the left-hand sides of (\ref{glueing_eq}) and (\ref{glueing_eq_logs}).


\subsection{The log-curvature map}
\label{sec:log-curvature map}

Let $\H$ be the upper half plane and $Z = \{ z \in \H^{\Box} | z(q') = \frac{1}{1-z(q)} \}$.\\
The \emph{log-curvature map} $G\co Z \rightarrow \C^E$ is defined by:
$$
G(z)(e) = \sum_{q \in \Box} i(q,e) \log(z(q)).
$$

The log-curvature map is, of course, closely related to the complex-curvature map $c \co Z \rightarrow \C^E$
$$
c(z)(e) = \prod_{q \in \Box} z(q)^{i(q,e)}.
$$
For instance, a well-known Euler characteristic argument shows that if the link of each vertex is a torus, then $G^{-1}(2\pi i, \ldots, 2\pi i) = c^{-1}(1, \ldots, 1).$ In general, $c^{-1}(u)$ is a countable union of sets of the form $G^{-1}(u_k),$ and these sets are pairwise disjoint by analytic continuation. But since $c^{-1}(u)$ is an affine algebraic set, at most finitely many of these sets will be non-empty. Hence each level set $G^{-1}(u_k)$ of the log-curvature function is also an affine algebraic set.


\section{Statements and proofs of the main results}

Our central result is Theorem~\ref{injective_thm}, generalising Choi's main technical result \cite[Theorem 4.13]{ch} to pseudo-manifolds with positively oriented triangulations and prescribed log-curvature.


\subsection{Rank of the log-curvature map}

Fixing a preferred quadrilateral type $q < \sigma$ for every 3--simplex, there is a natural identification between $Z$ and $\H^T$ via the projection map $\pi\co Z \rightarrow \H^T$, where
$$
\pi(z)(\sigma) = z(q).
$$
Hence the log-curvature map can and will be viewed as a map $G: \H^T \rightarrow \C^E$. The following result is a corollary of \cite[Theorem 4.1]{ne}:

\begin{thm}[Neumann] \label{thm_dimension}
$\diff G$ has constant rank $\displaystyle |T| - \sum_{v\in V} g_v$.
\end{thm}

\begin{proof}
For each tetrahedron $\sigma_i \in T$, fix a quadrilateral type $q_i<\sigma_i$. Using the quadrilateral index, we let
$B=(b_{ij})$ be the $2|T| \times |E|$ matrix, where for all even $i$:
$$
\begin{cases}
b_{ij} = i(q_i,e_j) - i(q_i',e_j),\\
b_{i+1 j} = i(q_i',e_j) - i(q_i'',e_j).
\end{cases}
$$
It is well known that $\diff G$ has the same rank as $B$ (see \cite{NZ}), hence it is enough to show that $\rank B = |T| - \sum_{v\in V} g_v$. This is the contents of \cite[Theorem, 4.1]{ne}. Neumann defines a linear map $\beta: C_1 \rightarrow J$, where $C_1$is the free $\Z$--module generated by $E$ and $J$ is the free $\Z$--module generated by $T^2$, and shows that 
$$
\dim( \Ker \beta^* / \ima \beta ) = \dim \bigoplus_{v \in V} \Homo_1(\Lk(v) ) = \displaystyle \sum_{v \in V} 2 g_v,
$$
where $\beta^*$ is the dual map. Moreover, Neumann shows $\Ker \beta^* = 2|T| - |E| + |V|$, and therefore
$$
\dim( \ima \beta ) = \dim( \Ker \beta^*) - \dim \bigoplus_{v \in V} \Homo_1(\Lk(v) ) = 2|T| - |E| + |V| - \sum_{v \in V} 2 g_v.
$$
Using particular bases of $C_0$ and $J$, $B$ is the matrix associated to $\beta,$ giving
$$
\rank B = \dim( \ima \beta ) = 2|T| - |E| + |V| - \sum_{v \in V} 2 g_v = T - \sum_{v \in V}  g_v,
$$
where the last equality follows from Lemma \ref{dimension_count}.
\end{proof}

Theorem \ref{thm_dimension} and the implicit function theorem imply the following result.

\begin{cor}\label{for:level set is smooth}
 For all $u \in \C^E$, the complex variety $G^{-1}(u)$ is either empty or a smooth complex manifold of dimension $\displaystyle \sum_{v \in V}  g_v$.
\end{cor}


\subsection{Boundary Map}
Let $\alpha$ be an oriented closed normal curve on $\Lk(v)$, representing a non--trivial element in $H_1(\Lk(v))$ and let $t \in \tri_v^{(2)}$ be a normal triangle contained in tetrahedron $\sigma \in T$. 
A segment of $\alpha$ with respect to $t$ is an oriented connected component of $\alpha \cap t$. Let $S_{\alpha}^t$ be the set of all segments of $\alpha$ with respect to $t$, and $\displaystyle S_{\alpha} = \bigcup_{t \in \tri_\partial} S_{\alpha}^t$ be the set of all segments of $\alpha$.

\begin{figure}[ht]
\centering
\includegraphics[width=0.3\textwidth]{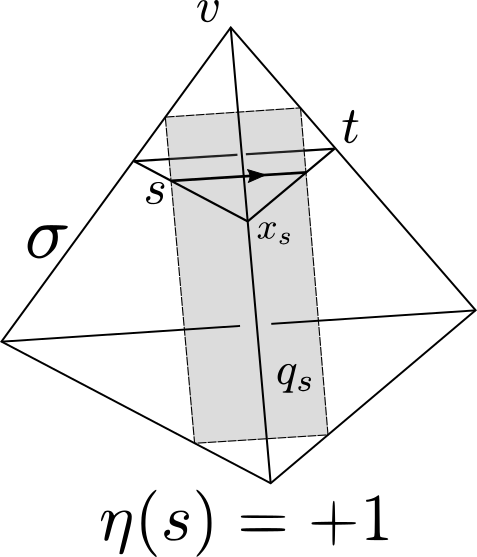} \qquad \qquad   \includegraphics[width=0.3\textwidth]{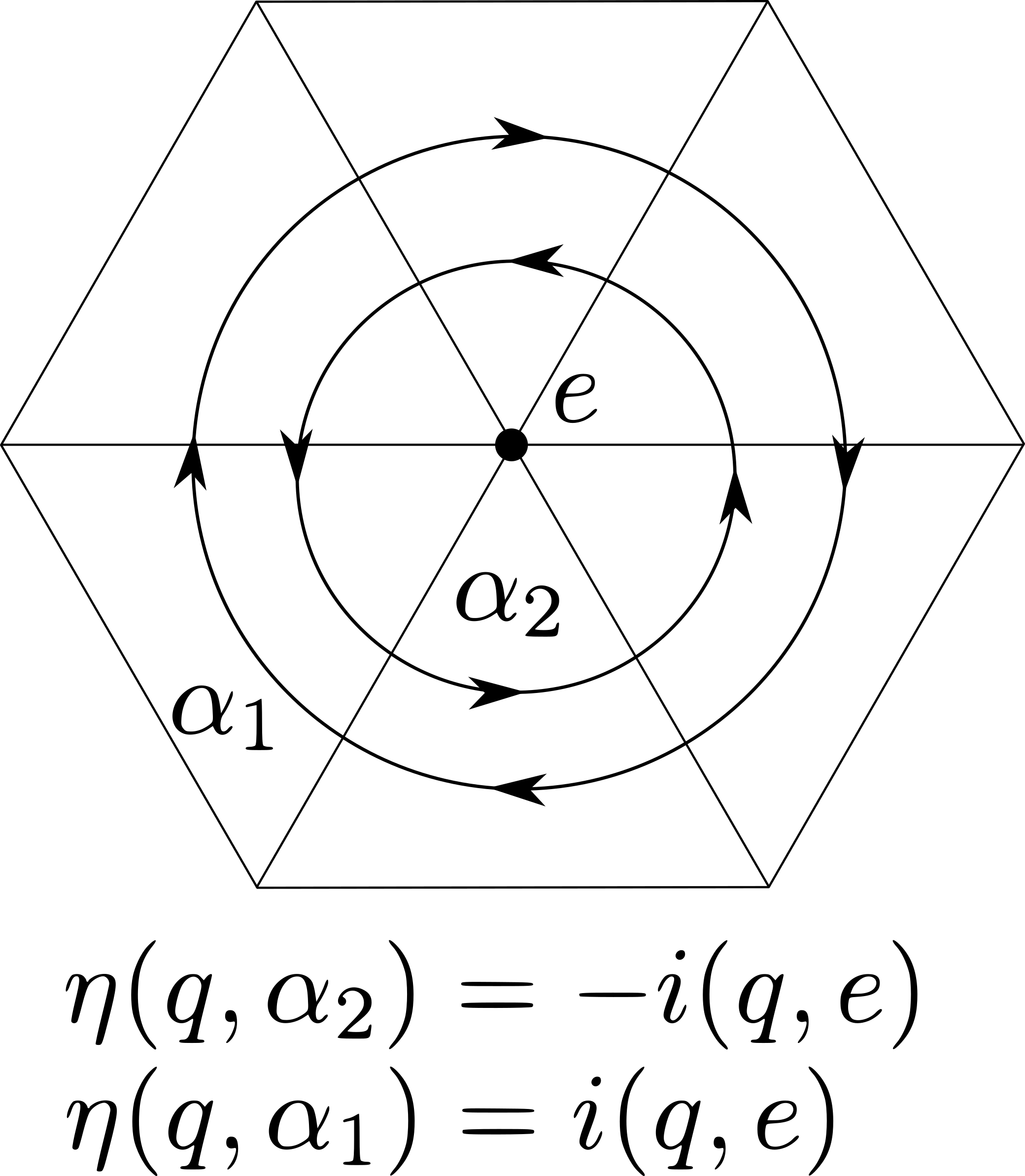}
\caption{On the left: Each $s \in S_\alpha$ determines a quadrilateral type $q_s$ and isolates a vertex $x_s$. For an observer sitting in the cusp, $x_s$ here lies on the right side of $s$. On the right: $\alpha_1,\alpha_2$ encircle the endpoint of an edge $e \in E$.}
\label{tetrahedron_norm_quad}\label{alpha_around_edge}
\end{figure}


Each $s \in S_\alpha$ uniquely determines a quadrilateral type $q \in \Box$ such that $s \subset q$, which will be denoted by $q_s$. If $s \in S^t_\alpha$, $s$ divides $t$ in two regions and isolates one of the three vertices, say $x_s$ (Figure \ref{tetrahedron_norm_quad}). With respect to the orientation of $s$ and the induced orientation on $\Lk(v)$, it makes sense to say that $x_s$ lies on the right side or on the left side of $s$ when viewed from the cusp. We therefore define
$$
\ind(s) =
\begin{cases}
+1 & \text{ if } x_{s} \text{ lies on the right side of } s, \\
-1 & \text{ if } x_{s} \text{ lies on the left side of } s,
\end{cases}
$$
$$
\ind(q,s) =
\begin{cases}
\ind(s) & \text{ if } q = q_s, \\
0 & \text{ otherwise },
\end{cases}
$$
and finally
$$
\ind(q,\alpha) = \sum_{s \in S_\alpha}\ind(q,s).
$$

In particular, when $\alpha$ encircles one endpoint of an edge $e \in E$, then $\eta(q,\alpha) = \pm i(q,e)$ for all $q\in \Box$, where the sign depends on whether $\alpha$ is oriented anticlockwise or clockwise with respect to the cusp (Figure \ref{alpha_around_edge}).

Hence we can write the log holonomy of $\alpha$ as $\h_{\alpha}: Z \rightarrow \C $,
$$
\h_{\alpha}(z) = \sum_{q \in \Box}\ind(q,\alpha) \log z(q).
$$
Notice that if $\alpha_1$ is normally isotopic to $\alpha_2$, then $\h_{\alpha_1}(z) = \h_{\alpha_2}(z)$, and if $\alpha_1$ is the union of $n$ normal curves all normally isotopic to $\alpha_2$, then $\h_{\alpha_1}(z) = n\h_{\alpha_2}(z).$ We therefore have a natural 
extension of the log holonomy to $\n_{\RR}$, the $\RR$--vector space with basis the normal isotopy classes of curves on $\Lk(v)$. For all $a,b \in \RR$ and $\alpha,\beta \in \n_{\RR}$ define
$$
\h_{a \alpha + b \beta}(z) := a \h_{ \alpha}(z) + b \h_{ \beta}(z).
$$
Note that if $\alpha^{-1}$ is the closed normal curve $\alpha$ with opposite orientation, then $\h_{\alpha + \alpha^{-1}}(z) = 0$. 

If $G(z) \not= (2\pi i,\dots,2\pi i) $, then the \emph{value} of $\h_{\alpha}(z)$ may not be an invariant of the homology class of $\alpha,$ but depends on the choice of normal representative since an isotopy pushing $\alpha$ over a vertex in $\Lk(v)$ may change its value. 

For each link $\Lk(v_i),$ let $g_i$ denote its genus and choose a canonical homotopy group generating set $\l_i \cup \m_i,$ where $\l_i = \{ \lambda^i_1,\dots,\lambda^i_{g_i} \}$ is the \emph{set of longitudes} and $\m_i = \{ \mu^i_1,\dots,\mu^i_{g_i} \}$ the \emph{set of meridians}, indexed and oriented such that the algebraic intersection number satisfies $\iota(\lambda^i_j,\mu^i_j)=1$ for all $j$ and $\iota(\alpha,\beta)=0$ for all other pairs of elements $\alpha, \beta \in \l_i \cup \m_i.$ Moreover, we assume that each element of $\l_i \cup \m_i$ is an oriented normal curve on $\Lk(v_i).$

Let $\l = \bigcup_i \l_i$ and define the \emph{boundary map} $H_{\l}: Z \rightarrow \C^{\l} $ by
$$
H_{\l}(z)(\lambda) = \h_{\lambda}(z) \qquad \mbox{ for all } \  \lambda \in \l.
$$
Similarly, we let $\displaystyle \m = \bigcup_i \m_i$ and $\displaystyle \Lk(V) = \bigsqcup_i \Lk(v_i)$. Hence  $\l\cup\m$ is a generating set for $H_1(\Lk(V))$.


\subsection{Tangential angle structures}

Our study of derivatives naturally leads us (at least implicitly) to the tangent space of the space of all angle structures on $\tri.$ Following \cite{Lu}, the space of all \emph{tangential angle structures} $\TAS=\TAS(\tri)$ is the set of all $\alpha \in \RR^{\Box}$ such that
\begin{itemize}
\item $\displaystyle \sum_{q < \sigma} \alpha(q) = 0$, $ \forall\sigma \in T$, and
\item $\displaystyle\sum_{q \in \Box} i(q,e)\alpha(q)=0 $, $\forall e \in E$.
\end{itemize}
The endgame of the proof of Theorem~\ref{injective_thm} uses a specific spanning set of
$$
\STAS = \STAS(\tri) = \{ w \in \TAS  \ | \ \sum_{q \in \Box}\ind(q,\lambda) w(q) = 0 \ \forall \lambda \in \l \} \le \TAS,
$$
which we will now determine. For every edge $e \in E$ and every normal curve $\gamma$ on a vertex link, let $Q_e,Q_{\gamma} \in \RR^{\Box}$ be defined by
\begin{align*}
&Q_e = \sum_{q:q\sim e} (q')^* - (q'')^*,\\
&Q_{\gamma} = \sum_{q \in \Box}\ind(q,\gamma) \left(\;(q')^* - (q'')^*\; \right).
\end{align*}
In their study of angle structure on cusped manifolds,
Futer and Gu\'eritaud~\cite[Section $4$]{FuGu} introduced $Q_e$ and $Q_\gamma$ under the names of \emph{leading--trailing deformation around $e$} and \emph{leading--trailing deformation along $\gamma$} respectively. 
In fact, for every edge $e$, one may choose a normal closed curve $\rho$ about one endpoint of $e$ and an orientation of $\rho$ such that $Q_e = Q_{\rho}.$ Hence $Q_{e}$ can be thought of as a leading--trailing deformation along $\rho.$
Some of the results of \cite{FuGu} extend directly to the more general setting of an oriented pseudo-manifold.
As in \cite[Lemma 4.5]{FuGu}, we have
$
Q_e, Q_{\gamma}\in \TAS.
$

\begin{lem} \cite[Lemma 4.4]{FuGu} \label{FuGu_lem}
 Let $\alpha,\beta$ be oriented  closed normal curves on $\Lk(V)$ that intersect transversely, if at all. Then
$$
\frac{\partial}{\partial Q_\beta} \im( \h_{\alpha} ) =\sum_{q \in \Box}\ind(q,\alpha) Q_{\beta}(q) = 2 \ \iota(\alpha,\beta).
$$
\end{lem}

\begin{proof}
The proof in \cite[Lemma 4.4]{FuGu} does not use the fact that $\partial M$ is a union of tori (and in particular that in this case $\h_{\alpha}$ is independent of the choice of normal curve in a homology class) and applies verbatim. 
\end{proof}

\begin{lem} \label{dimension_Qe}
The set $\{Q_e\}_{e \in E}$ spans a subspace of dimension $|T| - \sum_i g_i$.
\end{lem}

\begin{proof}
For each tetrahedron $\sigma_j \in T$ fix a normal quadrilateral type $q_j < \sigma_j$. Let $B$ be the $( |E| \times 3|T| )$ matrix whose $i$--th row is the vector $Q_{e_i}$, namely
$$
\left( Q_{e_i}(q_1), Q_{e_i}(q_1'), Q_{e_i}(q_1'')\dots,Q_{e_i}(q_{|T|}), Q_{e_i}(q_{|T|}') , Q_{e_i}(q_{|T|}'') \right),
$$
where $e_i \in E$. For every vertex $v \in V$, we define a row vector $r_v \in \RR^{|E|}$, whose $i$--th entry is the number of endpoints that the $i$--th edge $e_i$ has at the vertex $v$.\\
In \cite[Lemma 3.3]{FuGu}, a matrix $A$ and vectors $r_c$ were constructed in a similar fashion, and they are related to $B$ and $r_v$ as follows:
\begin{itemize}
	\item Let $\sigma \in S_3$ be the permutation $(123)$ and $A_i,B_i$ be the $i$-th columns of $A,B$ respectively. Then for all $0\leq k \leq |T|$ and $1\leq i \leq 3$, $B_{3k + i}$ is the column vector made up of the last $|E|$ entries of $A_{3k + \sigma^{-1}(i)} - A_{3k + \sigma(i)}$;
	\item $r_v$ is the row vector made of the last $|E|$ entries of $r_c$.
\end{itemize}
Following the proof of \cite[Lemma 3.3]{FuGu}, one checks that the vectors $r_v$ form a basis for the row null space of $B$ and therefore $\rank (B) = |E| - |V| $.
 The conclusion follows from Lemma \ref{dimension_count}.
\end{proof}

The following result generalises \cite[Proposition 4.6]{FuGu}. To simplify notation, we write $\{Q_{\lambda}\}= \{Q_{\lambda}\}_{\lambda \in \l}$ etc.

\begin{lem} \label{FuGu_prop}
 The set $\{Q_{\lambda}\}\cup \{Q_{\mu}\}$ is linearly independent, and $\Span \left( \{Q_{\lambda}\} \cup \{Q_{\mu}\} \right) \cap \Span \{Q_e\} = \{ 0 \}$. In particular,
\begin{itemize}
 \item[$i)$] $\dim \Span \{Q_{\mu}\}  = \sum_i g_i,$
 \item[$ii)$] $\dim \Span \left( \{Q_e\} \cup \{Q_{\lambda}\} \right)= |T|,$
 \item[$iii)$] $\Span \left(\{Q_e\}\cup \{Q_{\lambda}\} \cup \{Q_{\mu}\} \right) = \TAS(\tri)$.
\end{itemize}
\end{lem}

\begin{proof}
	For every edge $e$, choose a normal closed curve $\rho$ about one endpoint of $e$. Then we can choose an orientation on $\rho$ such that $Q_e = Q_{\rho}$ and think of $Q_{e}$ as a leading--trailing deformation along $\rho$.\\
	Let $I : \Span \left( \{Q_{\lambda}\} \cup \{Q_{\mu}\} \cup \{Q_{\rho}\} \right)  \longrightarrow \RR^{|\l|+|\m|}$ be the map defined as follows. If $ X = \sum a_i  Q_{\lambda_i}+ \sum b_j Q_{\mu_j} + \sum c_k Q_{\rho_k}$, set $\xi = \sum a_i  \lambda_i+ \sum b_j \mu_j + \sum c_k \rho_k$ and define
	$$
	I(X) =
	\left( \frac{\partial}{\partial Q_{\lambda_1}} \im( \h_{\xi} ), \dots,  \frac{\partial}{\partial Q_{\lambda_{|\l|}}} \im( \h_{\xi} ), 	\frac{\partial}{\partial Q_{\mu_1}} \im( \h_{\xi} ), \dots,  \frac{\partial}{\partial Q_{\mu_{|\m|}}} \im( \h_{\xi} ) \right).
	$$
	
	The linearity of $\h_v$ implies that $I$ is linear, and by Lemma \ref{FuGu_lem}, we have
	\begin{itemize}
		\item $ I(Q_{\rho}) = (0,\dots,0),$ $i.e.$ $\Span\{Q_e\}  \subset \Ker I$;
		\item $I(Q_{\lambda_i})$ has $1$ in the $(|\l |+i)$--th entry and $0$ everywhere else,\newline $i.e.$ maps to the $(|\l |+i)$--th standard basis vector.
		\item $I(Q_{\mu_j})$ has $1$ in the $j$--th entry and $0$ everywhere else, \newline $i.e.$ maps to the $j$--th standard basis vector.
	\end{itemize}	
	This implies $ \Span \left( \{Q_{\lambda}\} \cup \{Q_{\mu}\} \right)   \cong \ima I = \RR^{|\l|+|\m|},$ so
	$\{Q_{\lambda}\}_{\lambda \in \l}\cup \{Q_{\mu}\}_{\mu \in \m}$ is linearly independent, and 
	$\Span \left( \{Q_{\lambda}\} \cup \{Q_{\mu}\} \right) \cap \Span \{Q_e\}  = \{ 0 \}.$ 
	This implies $(i)$, and together with Lemma~\ref{dimension_Qe}, it implies $(ii).$
	By \cite[Corollary 2.3]{Lu}, we have
$$
\dim \TAS = |V| - |E| + 2|T| = |T| + \sum_i g_i.
$$
	It now follows from Lemma~\ref{dimension_Qe} that
	$\dim \Span \left(\{Q_e\}\cup \{Q_{\lambda}\} \cup \{Q_{\mu}\} \right) = \dim \TAS,$
	 which implies $(iii).$
 	We also note that this shows $\Span\{Q_e\} = \Ker I.$
\end{proof}

\begin{thm} \label{Q_in_STAS}
The set $ \{Q_e\} \cup \{Q_{\lambda}\}$ is a subset of $\STAS$, whereas $\Span ( \{Q_{\mu} \} ) \cap \STAS = \{0\}$. Therefore $\Span \left( \{Q_e\} \cup \{Q_{\lambda}\} \right) = \STAS$.
\end{thm}

\begin{proof}
 Since $ \{Q_e\} \cup \{Q_{\lambda}\}$ is a subset of $\TAS$, the first part follows from checking that for all $\lambda,\lambda' \in \l$ and $e\in E$
\begin{align}
 &\sum_{\overline{q} \in \Box }\ind(\overline{q},\lambda') Q_e (\overline{q}) = 0,\\
 &\sum_{\overline{q} \in \Box}\ind(\overline{q},\lambda') Q_{\lambda}(\overline{q}) = 0. \label{STAS_property}
\end{align}
Observe that for all $q,\overline{q} \in \Box$,
$$
(q')^*(\overline{q}) - (q'')^*(\overline{q})  = - \left( \ (\overline{q}')^*(q) - (\overline{q}'')^*(q) \ \right),
$$
hence
\begin{align*}
\sum_{\overline{q} \in \Box}\ind(\overline{q},\lambda') Q_e (\overline{q}) & = \sum_{\overline{q} \in \Box}\ind(\overline{q},\lambda') \sum_{q \in \Box} i(q,e) \left( (q')^*(\overline{q}) - (q'')^*(\overline{q}) \right)\\
& = \sum_{q \in \Box} i(q,e) \sum_{\overline{q} \in \Box}\ind(\overline{q},\lambda')  \left( (q')^*(\overline{q}) - (q'')^*(\overline{q}) \right)\\
& = - \sum_{q \in \Box} i(q,e) \sum_{\overline{q} \in \Box}\ind(\overline{q},\lambda')  \left( \ (\overline{q}')^*(q) - (\overline{q}'')^*(q) \ \right)\\
& = - \sum_{q \in \Box} i(q,e) Q_{\lambda'} (q) = 0,
\end{align*}
as $Q_{\lambda'} \in \TAS$. This shows $Q_e \in \STAS$.

Now suppose $\lambda,\lambda' \in \l$, then $\iota(\lambda', \lambda) = 0$ as they are disjoint by assumption. By Lemma \ref{FuGu_lem}, this shows that $\{Q_{\lambda}\}$, and in particular $ \{Q_e\} \cup \{Q_{\lambda}\}$, is a subset of $\STAS$, and by Lemma~\ref{FuGu_prop}$(ii)$,
$$
\dim \STAS \geq |T|.
$$

For the next part, let $\sum_{i,j} a^i_j Q_{\mu^i_j}$ be an element in $\Span ( \{Q_{\mu} \} )$ and suppose by contradiction $\sum_{i,j} a^i_j Q_{\mu^i_j} \in \STAS$. Then for all $l$ and for all $\lambda^l_k \in \l_l$, by Lemma \ref{FuGu_lem},
$$
0 = \sum_{q \in \Box}\ind(q,\lambda^l_k) \left( \sum_{i,j} a^i_j Q_{\mu^i_j}(q) \right) = \sum_{i,j} a^i_j \ \iota(\lambda^l_k,\mu^i_j).
$$
But $\iota(\lambda^l_k,\mu^i_j) \not= 0$ if and only if $i = l$ and $k=j$, therefore $a^l_k = 0$ for all $l,k$ and 
$$
\Span ( \{Q_{\mu} \} ) \cap \STAS = \{0\}.
$$ 
It follows from Lemma~\ref{FuGu_prop}$(i)$ that $\dim \STAS \leq |T|$. Hence $\dim \STAS = |T|$ and so 
\begin{equation*}
\Span \left( \{Q_e\} \cup \{Q_{\lambda}\} \right) = \STAS.\qedhere
\end{equation*}
\end{proof}


\subsection{A Parametrization of $G^{-1}(u)$}
\label{sec:parametrisation}

Using the identification $Z = \H^T,$ we write $H_{\l}: \H^T \rightarrow \C^{\l},$ giving the map
\begin{align*}
 (G,H_{\l}): \H^T &\to \quad \C^E \times \C^{\l}\\
 z \quad &\mapsto \quad \left(G(z),H_{\l}(z) \right).
\end{align*}

We already remarked that when $G(z) \not= (2\pi i,\dots,2\pi i) $, the \emph{value} of the boundary map $H_{\l}$ depends on the choice of normal curves representing homology classes of peripheral curves. The following lemma shows that the \emph{injectivity} of the differential of $(G,H_{\l})$ is independent of this choice.

\begin{lem}
For each vertex $v_i \in V$, let $g_i$ be the genus of $\Lk(v_i)$. Let $A_i = \{ \alpha^i_1,\dots,\alpha^i_{g_i} \}$ and $B_i = \{ \beta^i_1,\dots,\beta^i_{g_i} \}$ be two sets of longitudes on $\Lk(v_i)$, such that $\alpha^i_j$ and $\beta^i_j$ are representatives of the same element in $H_1(\Lk(v_i))$, and set $A = \bigcup_i A_i$ and $B = \bigcup_i B_i$. Then
$$
 \rank \diff(G,H_{A}) = \rank \diff(G,H_{B}).
$$
In particular $\diff(G,H_{A})$ is injective if and only if $\diff(G,H_{B})$ is injective.
\end{lem}

\begin{proof}
 As $\alpha^i_j$ and $\beta^i_j$ are in the same homotopy class,
$$
\h_{\alpha^i_j}(z) = \h_{\beta^i_j}(z) + \sum_{e \in E} a_e G(z)(e) \qquad \mbox{ for some } a_e \in \RR.$$
Hence
$$
\nabla \h_{\alpha^i_j}(z) = \nabla  \h_{\beta^i_j}(z) + \sum_{e \in E} a_e \nabla G(z)(e),
$$
$i.e.$ each row of $\diff(H_A)$ is a linear combinations of the corresponding row of $\diff(H_B)$ and rows of $\diff(G)$. It follows that $\diff(G,H_{A})$ is related to $\diff(G,H_{B})$ by elementary row operations, hence they have the same rank.
\end{proof}

\begin{thm} \label{injective_thm}
The derivative $\diff(G,H_{\l}): T_z(\H^T) \rightarrow \C^E \times \C^{\l}$ is injective for any $z \in \H^T$.
\end{thm}

\begin{proof}
Let $z \in Z$. For every $q \in \Box$, $z(q) \in \H$ uniquely determines three angles $\alpha, \beta$ and $\gamma$ of the triangle with vertices $(0,1,z(q))$ such that $z(q) = \frac{\sin(\beta)}{\sin(\gamma)} e^{i \alpha}$ and
$$
z(q') = \frac{\sin(\gamma)}{\sin(\alpha)} e^{i \beta} \qquad z(q'') = \frac{\sin(\alpha)}{\sin(\beta)} e^{i \gamma}.
$$
Hence we can identify $Z$ with the set $X = \{ x \in \RR_{>0}^{\Box} | \sum_{q<\sigma} x(q) = \pi \}$ via the map $\phi^T: X \rightarrow Z$, where
$$
\phi^T(x)(q) =  \frac{\sin(x(q'))}{\sin(x(q''))} e^{i x(q)}.
$$
Under this identification, we can write $(G,H_{\l})$ as a map $(G,H_{\l}) : X \rightarrow \C^E\times \C^{\l} = (\RR \times \RR)^{E\cup{\l}}$ where, for each edge $e\in E$ and curve $\lambda \in {\l}$,
\begin{align*}
G(x)(e) &= \sum_{q:q\sim e} \log \left(\frac{\sin(x(q'))}{\sin(x(q''))} e^{i x(q)} \right)\\
&= \left(\sum_{q:q\sim e} \log(\sin(x(q'))) - \log(\sin(x(q'')), \sum_{q \in \Box} i(q,e) x(q) \right).\\
H_{\l}(x)(\lambda) &= \sum_{q \in \Box}\ind(q,\lambda) \log \left(\frac{\sin(x(q'))}{\sin(x(q''))} e^{i x(q)} \right)\\
&= \left(\sum_{q \in \Box}\ind(q,\lambda) \left( \log(\sin(x(q'))) - \log(\sin(x(q'')) \right), \sum_{q \in \Box}\ind(q,\lambda) x(q) \right).
\end{align*}
Now our goal is to show that $\diff(G,H_{\l}): T_x X \rightarrow (\RR \times \RR)^{E\cup{\l}}$ is injective for all $x \in X$.
Let $w \in T_x X = \{ w\in \RR^{\Box} | \sum_{q < \sigma} w(q) = 0 \ \forall \sigma \in T\}$ be a tangent vector such that $\diff(G,H_{\l})(x)(w) = 0$.

We recall the elementary fact that $\frac{\diff}{\diff t} \log( \sin t) = \cot t$. For every edge $e$,
\begin{equation} \label{cot_condition_G}
\sum_{q:q\sim e} \cot(x(q'))w(q')  - \cot(x(q''))w(q'') = 0,
\end{equation}
\begin{equation} \label{tas_property_G}
\sum_{q \in \Box} i(q,e) w(q) = 0.
\end{equation}
and for every curve $\lambda$,
\begin{equation} \label{cot_condition_H}
\sum_{q \in \Box}\ind(q,\lambda) \left( \cot(x(q'))w(q')  - \cot(x(q''))w(q'') \right) = 0,
\end{equation}
\begin{equation} \label{tas_property_H}
\sum_{q \in \Box}\ind(q,\lambda) w(q) = 0.
\end{equation}
Observe that (\ref{tas_property_G}) and (\ref{tas_property_H}) imply $w \in \STAS \le \TAS \subset T_x X.$

Let $F : \RR^{\Box} \rightarrow \RR$ be the volume function
$$
F(x) = \sum_{q \in \Box} \Lambda(x(q)),
$$
where $\Lambda(x) = - \int_0^x \log |2\sin(u)|du$ is the Lobachevsky function.
The Hessian of $F$ is the diagonal matrix with diagonal entries $-\cot(x(q))$, and it is negative definite at each point $x\in X$ (see, for instance, \cite[pg. 17]{FuGu}). Recall that for $e \in E$ and $\lambda \in \l$, we have
,\begin{align*}
&Q_e = \sum_{q:q\sim e} (q')^* - (q'')^*,\\
&Q_{\lambda} = \sum_{q \in \Box}\ind(q,\lambda) \left((q')^* - (q'')^* \right).
\end{align*}
Then by (\ref{cot_condition_G}) and (\ref{cot_condition_H})
\begin{align*}
& Q_e \cdot \Hess_x(F) w = 0 \qquad \forall e \in E,\\
& Q_{\lambda} \cdot \Hess_x(F) w = 0 \qquad \forall \lambda \in \l.
\end{align*}
Theorem~\ref{Q_in_STAS} shows that $\STAS$ is spanned by $\{Q_e\}_{e \in E} \cup \{Q_{\lambda}\}_{\lambda \in \l}.$
Hence $w$ can be written as a linear combination of the $Q_e$ and $Q_\lambda$, and so
$$
w \cdot \Hess_x(F)w = 0.
$$
But since $\Hess_x(F)$ is negative definite, we have $w = 0$ and hence $\diff(G,H_{\l})$ is injective.
\end{proof}

Theorem \ref{thm_dimension} shows that the rank of $\diff G$ is constant $|T| - \sum_{v\in V} g_v$, hence Theorem \ref{injective_thm} implies the following result.

\begin{cor}\label{cor:parametrication}
 For all $u \in \C^E$ such that $G^{-1}(u) \not= \emptyset$, the restriction map
$$
\restr{H_{\l}}{G^{-1}(u)}: G^{-1}(u) \longrightarrow \C^{\l}
$$
is a local diffeomorphism onto its image.
\end{cor}


\section{Two Dehn-surgery components}

Our first example of a one-cusped hyperbolic 3--manifold of finite volume illustrates Choi's original result. It also has the property that the two discrete and faithful characters lie on different components of the $\PSL$--character variety. It was found via the following construction due to Nathan Dunfield. Let $N$ be a 2--cusped
manifold with strong geometric isolation (see \cite{NR} for a definition). Let $A$ be the collection of all
hyperbolic Dehn fillings on the first cusp in $N$; let $M_1$, $M_2$, $M_3$,... be
distinct manifolds in $A.$    Fix an orientation of $N,$ and let $\chi_0$ be the
character of an associated holonomy representation.  Let $\chi_n$ be the
character of the holonomy representation of $M_n$ that is in a small open neighbourhood of $\chi_0,$ and
let $X_n$ be the component of the character variety of $M_n$ which contains
$\chi_n$.   By strong geometric isolation, the image of $X_n$ in the character
variety of the second cusp is independent of $n$, call it $C.$   Now the
differential of the volume function on $X_n$ pulls back from a 1--form on $C.$
Hence, the quantity 
\begin{equation}
        \max\{ \vol(\chi) \mid \chi \in X_n\} - \min\{ \vol(\chi) \mid \chi \in X_n\} 
\end{equation}
is a constant $V$ independent of $n.$  By volume rigidity, $V \le 2 \min(
\vol(M_n) )$.   So if $M_n$ has non-minimal volume (and of course $\vol(M_n) \to
\vol(N)$ so there are many such), it follows that $\chi_n$ and its complex
conjugate must lie on different components of the character variety of $M_n.$

The following explicit example was found by looking in the census due to Callahan, Hildebrand and Weeks (as shipped with {\tt Regina }\rm \cite{Regina}) for an example given by Neumann and Reid~\cite{NR}.
Let $M$ be the manifold $t12046$ in this census, and let $N$ be the manifold obtained from $M$ by $(2,1)$-Dehn filling on the first cusp. We fix the following ideal triangulation $\tri$ on $N$:

\begin{table}[h]
\begin{center}
   	\begin{tabular}{ | l | l | l | l | l |} 
    \hline
    Tetrahedron & Face $012$ & Face $013$ & Face $023$ & Face $123$ \\ \hline
    0 & 4 (203) & 2 (321) & 6 (032) & 5 (120) \\ \hline
    1 & 3 (312) & 2 (012) & 4 (013) & 6 (031) \\ \hline
    2 & 1 (013) & 4 (213) & 5 (123) & 0 (310) \\ \hline
    3 & 6 (312) & 4 (012) & 5 (203) & 1 (120) \\ \hline
    4 & 3 (013) & 1 (023) & 0 (102) & 2 (103) \\ \hline
    5 & 0 (312) & 6 (012) & 3 (203) & 2 (023) \\ \hline
    6 & 5 (013) & 1 (132) & 0 (032) & 3 (120) \\ \hline
    \end{tabular}
\end{center}
\caption{The ideal triangulation $\tri$ on $N$.}
\end{table}

Let $z_i$ be a shape parameter of the tetrahedron $T_i$, with respect to the edge $(01)$, with orientation such that $z'_i$ is the parameter at $(02)$ and $z''_i$ is at $(03).$
A standard computation provides the two solutions
$$z^{(0)} = \left(i,\;\frac{1+i}{2},\;i,\frac{2+i}{2},\;\frac{3+i}{5},\;i\;,\frac{-1+i}{2}\right) \in \C^7$$
 and its complex conjugate $\overline{z^{(0)}}$ corresponding to discrete and faithful characters. Notice that $z^{(0)}$ is positively oriented.
 The one-dimensional components $C_0$ and $C_0'$ containing them have the following rational parametrizations:
 
 {\small
 $$
\varphi_0(u) =
\left(
\begin{array}{c}
\frac{-1-i + u}{-1 + (1 - i) u}\\ \\
\frac{1 + i}{(2 + 2 i) - 2 u}\\ \\
(1 + i) - \frac{i}{u}\\ \\
\frac{(1 - i) - 2 u + 2 u^2}{2 (-1 + u) u}\\ \\
\frac{(-1 + u) (-i + (1 + i) u)}{1 - (1 +i) u + (1 +i ) u^2}\\ \\
u\\ \\
\frac{(1 - i) + 2 i u}{(2 + 2 i) u - 2 u^2}\\ \\
\end{array}
\right) \qquad\text{and}\qquad
\varphi_0'(u) =
\left(
\begin{array}{c}
\frac{-1+i + u}{-1 + (1 + i) u}\\ \\
\frac{1 - i}{(2 - 2 i) - 2 u}\\ \\
(1 - i) + \frac{i}{u}\\ \\
\frac{(1 + i) - 2 u + 2 u^2}{2 (-1 + u) u}\\ \\
\frac{(-1 + u) (i + (1 - i) u)}{1 - (1 -i) u + (1 -i ) u^2}\\ \\
u\\ \\
\frac{(1 + i) - 2 i u}{(2 - 2 i) u - 2 u^2}\\ \\
\end{array}
\right)
$$
}

In particular, notice that $\varphi_0(u) = \overline{\varphi'_0(\overline{u})}.$

Henceforth, we will refer to the $j$-th components of  $\varphi_0(u)$ and $\varphi'_0(u)$ by $(\varphi_0)_j(u)$ and $(\varphi'_0)_j(u)$ respectively. The natural domain of $\varphi_0$ is $\C \setminus X$, where $X$ is the finite set of poles of $(\varphi_0)_j$. A direct computation shows that $X$ contains precisely $6$ elements, whose image via $\varphi_0$ are all ideal points of $C_0$, therefore
$$
\overline{\text{Im}(\varphi_0)} = \text{Im}(\varphi_0).
$$
Together with the image of $u= \infty$, they sum up to a total of $7$ ideal points, hence the image of $\varphi_0$ is a 7--punctured sphere. Moreover, a direct calculation reveals that the images of $\varphi_0$ and $\varphi'_0$ are disjoint.

\begin{figure}[ht]
\centering
\includegraphics[width=0.7\linewidth]{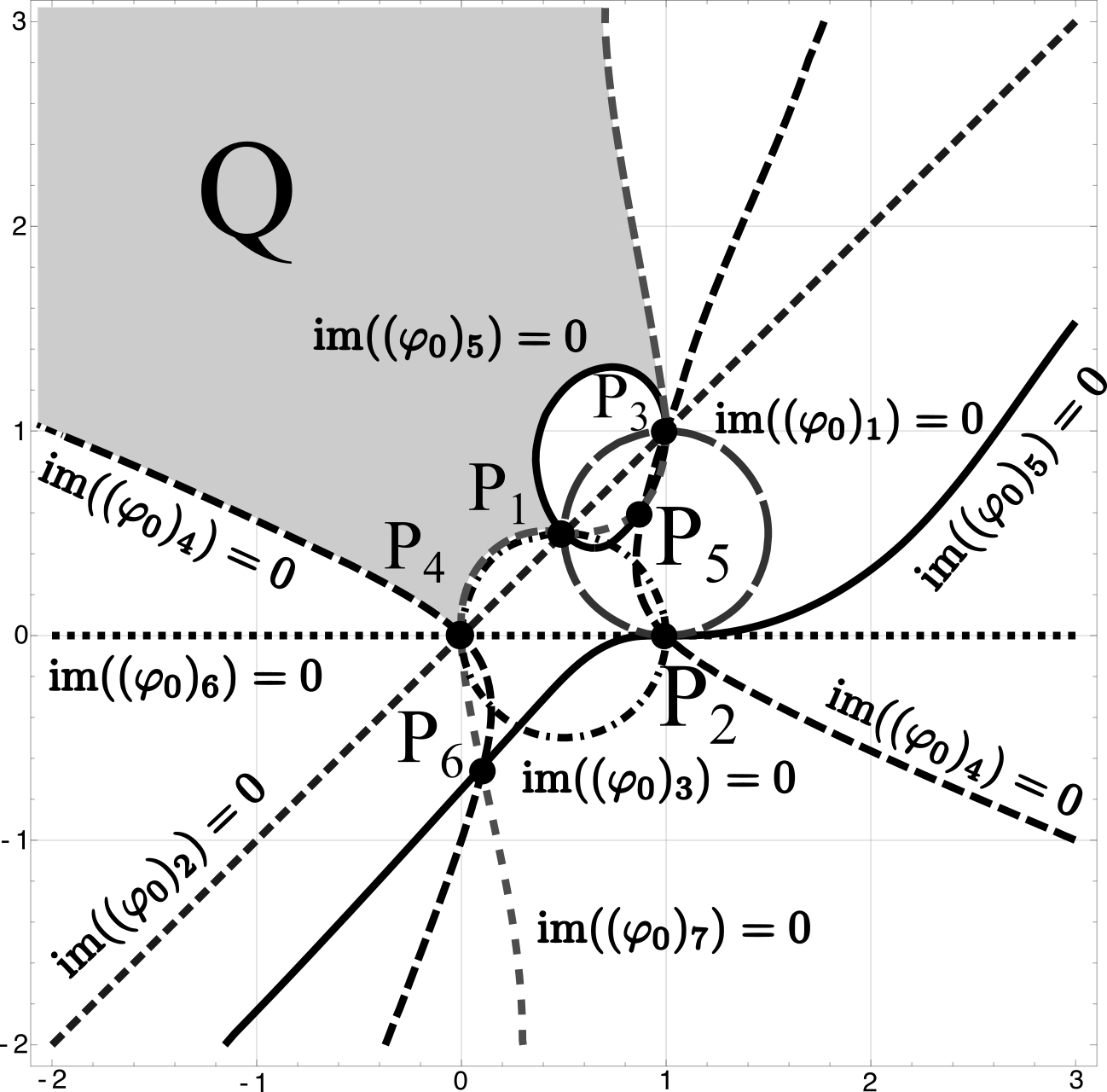}
\caption{Q is an ideal square.}
\label{square_ideal_region}
\end{figure}

Denoting by $\im(w)$ the imaginary part of a complex number $w \in \C$, we recall that the complete solution $\varphi_0(i) = z^{(0)}$ is geometric in the sense that
$$\im \big( (\varphi_0)_j(i) \big) = \im ( z^{(0)}_j ) > 0 \qquad \forall j \in \{1,\dots7\}.$$

The set of all positively oriented solutions is
$$
Q = \bigcap\limits_{j=1}^7 \{ u \in \C \ | \  \im \big( (\varphi_0)_j(u) \big) > 0 \}.
$$
 For $u = x + iy$,
$\{ \im \big( (\varphi_0)_j(u) \big) = 0 \}$ is a $1$-dimensional subvariety of $\RR^2,$ which can be explicitly computed from $\varphi_0$.

\begin{align*}
\im \big( (\varphi_0)_1 \big) & = (x-1)^2 + ( y- \frac{1}{2})^2 - \frac{1}{4},\\
\im \big( (\varphi_0)_2 \big) & = y-x,\\
\im \big( (\varphi_0)_3 \big) & = x^2 + y^2 - x, \\
\im \big( (\varphi_0)_4 \big) & = (y + \frac{1}{2})^2 - ( x \sqrt{3} - \frac{1}{2 \sqrt{3}})^2 - \frac{1}{6},\\
\im \big( (\varphi_0)_5 \big) & = -1+ 3x -3x^2 +x^3 - x^2y + y^2 + x y ^2 - y^3,  \\
\im \big( (\varphi_0)_6 \big) & = y, \\
\im \big( (\varphi_0)_7 \big) & = 2x - 3 x^2 - 2 x^3 + 2xy + y^2 - 2xy^2.  \\
\end{align*}

As shown in Figure \ref{square_ideal_region}, $Q$ is a simply connected open set of the plane and
$$
\partial Q  \subset V\left(\im \big( (\varphi_0)_4(u) \big),\im \big( (\varphi_0)_5(u) \big),\im \big( (\varphi_0)_7(u) \big) \right).
$$ 

We deduce that $\varphi_0(Q)$ is an ideal square contained in $\pi'(C_0).$ 


\section{A reducible complex-curvature level set}
		
	Let $M$ be the oriented pseudo--manifold given by the ideal triangulation $\tri$ shown in Figure~\ref{triangulation}. $M$ has two vertices $v_1$ and $v_2$, whose links are closed orientable surfaces of genus  $g_1 = 1$ and $g_2 = 2$ respectively. The induced triangulations of the links $\lk(v_1)$ and $\lk(v_2)$ are shown in the bottom right corner of Figure~\ref{triangulation} and in Figure~\ref{triangulation_cusp2}, respectively.

		\begin{figure}[ht]
			\centering
			\includegraphics[width=11cm]{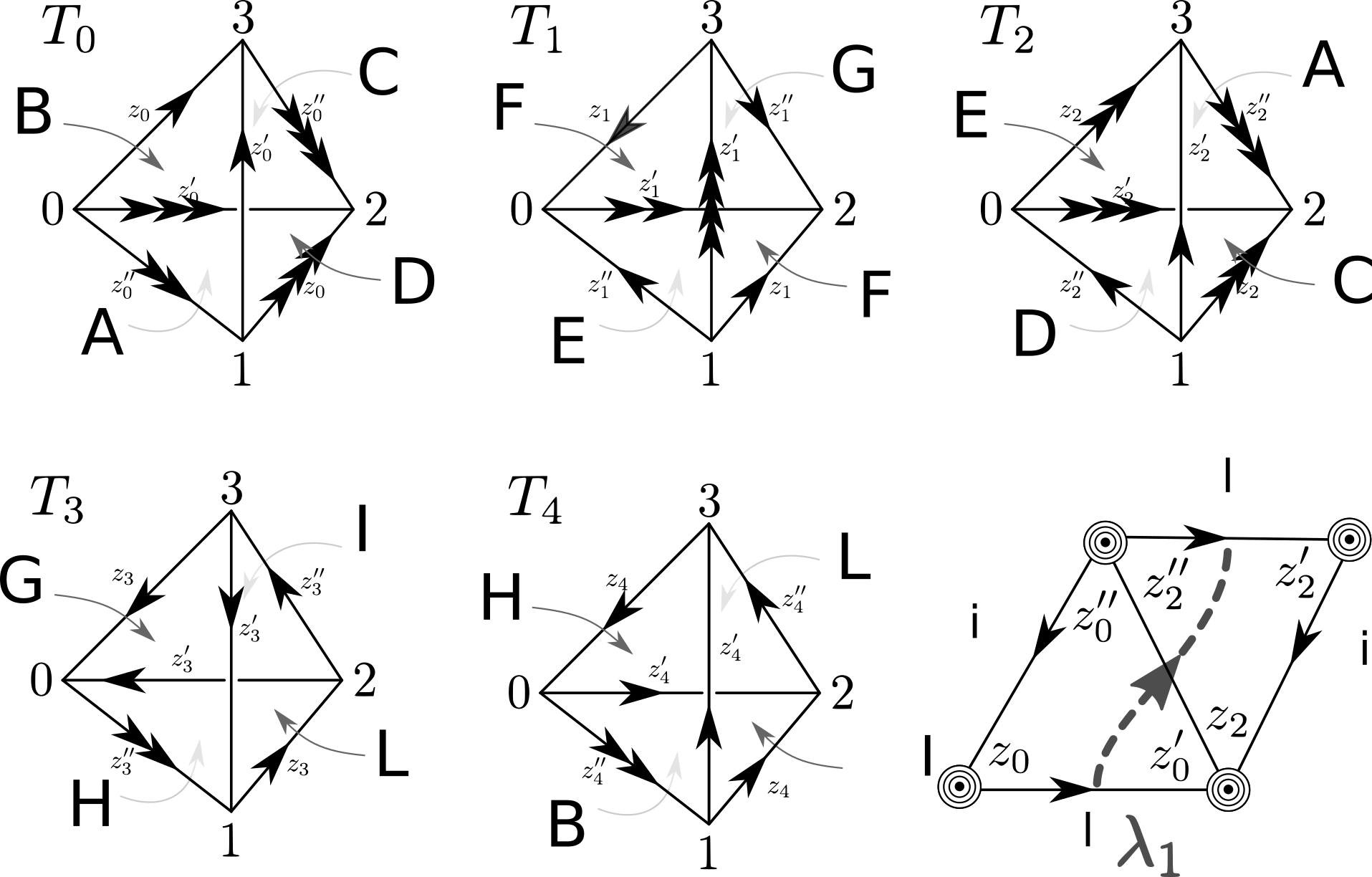}
			\caption{Shown are the five tetrahedra in the triangulation of $M$ and, in the bottom right corner, the induced triangulation of the vertex link of $v_1$ (viewed from the cusp).}
			\label{triangulation}
		\end{figure}
		
		\begin{figure}[ht]
			\centering
			\includegraphics[width=7cm]{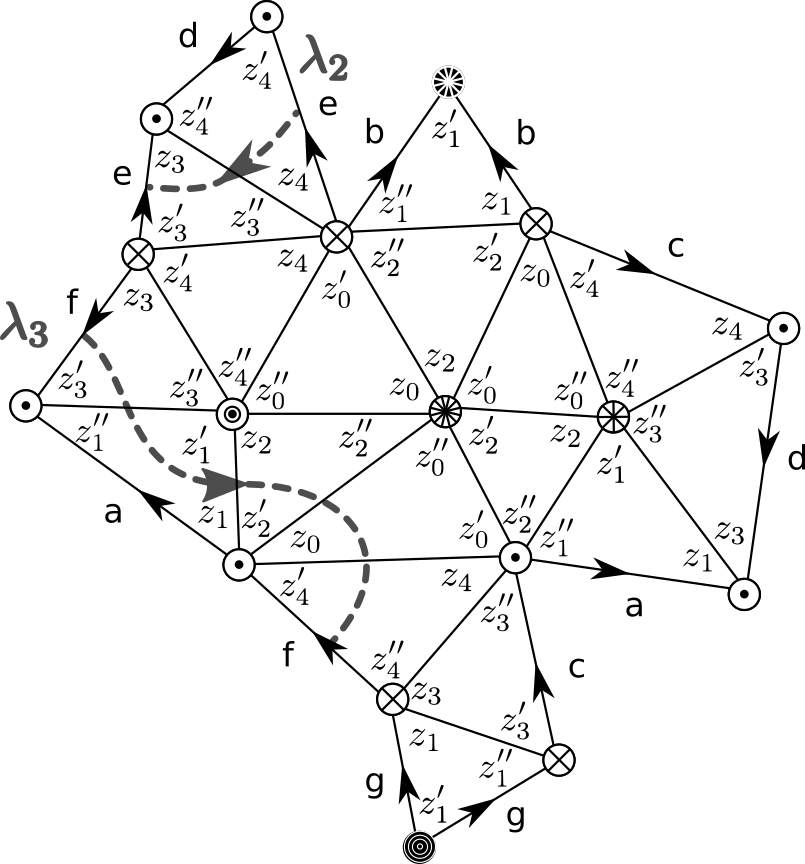}
			\caption{Triangulation induced on $v_2$ by $\tri$.}
			\label{triangulation_cusp2}
		\end{figure}
		
		\begin{table}[h]
			\begin{center}
				\begin{tabular}{ | l | l | l | l | l |} 
					\hline
					Tetrahedron & Face $012$ & Face $013$ & Face $023$ & Face $123$ \\ \hline
					0 & 2 (032) & 4 (012) & 2 (123) & 2 (120) \\ \hline
					1 & 2 (013) & 1 (213) & 3 (013) & 1 (103) \\ \hline
					2 & 0 (312) & 1 (012) & 0 (021) & 0 (023) \\ \hline
					3 & 4 (013) & 1 (023) & 4 (312) & 4 (230) \\ \hline
					4 & 0 (013) & 3 (012) & 3 (312) & 3 (230) \\ \hline
				\end{tabular}
			\end{center}
			\caption{The ideal triangulation $\tri$ on $N$.}
		\end{table}
	
We use the same notation as in the previous example, so $z_i$ is the shape parameter of tetrahedron $T_i$ with respect to edge $(12)$, and $z = (z_0,z_1,z_2,z_3,z_4)$. From the face pairings of $\tri$ we deduce the following complex-curvature and log-curvature maps:

	\begin{align}
	c(z) &= \left( 
	\begin{matrix}
	z_1'\\
	z_0'' z_1' z_2 z_3'' z_4''\\
	(z_0 z_0' z_0'') (z_2 z_2' z_2'')\\
	z_0 z_0' (z_1 z_1'')^2 z_2' z_2'' (z_3 z_3')^2 z_3'' (z_4 z_4')^2 z_4'' 
		\end{matrix} \right) = \left( 
	\begin{matrix}
	z_1'\\
	z_0'' z_1' z_2 z_3'' z_4''\\
	1\\
	z_0 z_0' (z_1 z_1'')^2 z_2' z_2'' z_3 z_3' z_4 z_4'
	\end{matrix} \right);\\
	G(z) &= \left(
	\begin{matrix}
		\log(z_1')\\
		\log(z_0'') + \log( z_1')+\log( z_2)+\log( z_3'')+\log( z_4'')\\
		\log(z_0 z_0' z_0'')+\log(z_2 z_2' z_2'')\\
		\log(z_0 z_0')+2\log( z_1 z_1'')+\log( z_2' z_2'')+\log( (z_3 z_3')^2 z_3'')+\log( (z_4 z_4')^2 z_4'')
	\end{matrix} \right) \label{G}
\end{align}
Notice that $\log(z-1) = \log(1-z) \pm \pi i$ where the sign ambiguity depends on the argument of $z$. However, applying the combinatorial Gauss--Bonnet theorem to the vertex, we deduce that
\begin{align*}
& G(z)(e_2)  = 2\pi i\\
& G(z)(e_0) + G(z)(e_1) + G(z)(e_3) = 8 \pi i
\end{align*}
Hence
{\scriptsize 
 $$
	G(z) = 
	\begin{pmatrix}
	-\log(1-z_1)\\ \\
	 \log(z_0-1)-\log(z_0) - \log(1-z_1)+\log( z_2)+\log( z_3-1)-\log(z_3) + \log( z_4-1)-\log(z_4)\\ \\
	2 \pi i\\ \\
	\log(z_0)-\log(z_0- 1) +2\log(1-z_1) - \log(z_2) +\log(z_3)-\log(z_3 - 1) + \log(z_4)-\log(z_4 - 1) + 8\pi i
	\end{pmatrix}.
$$
}
We observe that $\diff G$ has constant rank $|T| - g_1 - g_2 = 5-1-2 = 2$,
	$$
	\diff G(z) = \left(
	\begin{matrix}
	0 & \frac{-1}{z_1-1} & 0 & 0 & 0\\
	\frac{1}{z_0(z_0-1)} & \frac{-1}{z_1-1} & \frac{1}{z_2} & \frac{1}{z_3(z_3-1)} & \frac{1}{z_4(z_4-1)}\\
	0 & 0 & 0 & 0 & 0\\
	\frac{-1}{z_0(z_0-1)} & \frac{2}{z_1-1} & \frac{-1}{z_2} & \frac{-1}{z_3(z_3-1)} & \frac{-1}{z_4(z_4-1)}
	\end{matrix} \right).
	$$
	Moreover, for all $u = (u_0,u_1,u_2,u_3) \in \ima G$,  $G^{-1}(u)$ is the 3--dimensional variety generated by the ideal
	$$
	I = \big( z_1+e^{-u_0}-1, (z_0-1)z_2(z_3-1)(z_4-1) - e^{u_1}z_0(1-z_1)z_3z_4 \big).
	$$ 
	
For instance, let
$$
z^0 = \left( e^{\frac{\pi}{3}i},e^{\frac{\pi}{3}i},e^{\frac{\pi}{3}i},e^{\frac{\pi}{3}i},e^{\frac{\pi}{3}i} \right),
$$ 
then $u^0  = G(z^0) = (\frac{\pi}{3} i,\frac{5}{3}\pi i, 2 \pi i, 6\pi i) $ and
$$
G^{-1}(u^0) = \Big\{ (z_0,z_1,z_2,z_3,z_4) \in \H^5\ | \ z_1 = e^{\frac{\pi}{3} i}, z_2 = \frac{z_0z_3z_4 e^{\frac{4}{3}\pi i}}{(z_0-1)(z_3-1)(z_4-1)} \Big\}.
$$
	
For
$$
z^1 = \left(\frac{1}{1-e^{\frac{5}{6} \pi i}},e^{\frac{\pi}{3}i},\frac{1}{1-e^{\frac{5}{6} \pi i}},\frac{1}{1-e^{\frac{5}{6} \pi i}},\frac{1}{1-e^{\frac{5}{6} \pi i}}\right),
$$
$
u^1 = G(z^1) = \left(\frac{\pi}{3} i,\frac{11}{3}\pi i, 2 \pi i, 4\pi i\right)$, so
$$
G(z^0) \not= G(z^1) \ \mbox{ but } \ c(z^0) = c(z^1).
$$
By analytic continuation $G^{-1}(u^0)$ and $G^{-1}(u^1)$ are disjoint, hence the above argument shows that $G^{-1}(u^0)$ and $G^{-1}(u^1)$ are disjoint complex varieties contained in the complex variety $g^{-1}( c(z^0) )$. Moreover, because all shape parameters are assumed to have positive imaginary part, it follows from (\ref{G}) that $0 < G(z)(e_1)<5 \pi$ and hence
$$
g^{-1}( c(z^0) ) = G^{-1}(u^0) \bigcup G^{-1}(u^1).
$$
	
	Let $\l = \{\lambda_1, \lambda_2, \lambda_3 \}$ be the set of three simple closed curves shown in Figures \ref{triangulation} and \ref{triangulation_cusp2}. The boundary map with respect to $\l$ and its differential are

{\footnotesize
\begin{equation*}
e^{H_{\l}}(z) = \left(
\begin{matrix}
\frac{z_0'}{z_2''}\\ \\
\frac{z_3}{z_4}\\ \\
\frac{z_0 z_2' z_3' z_4'}{z_1'}
\end{matrix}  \right) \qquad
 H_{\l}(z) = \left(
  \begin{matrix}
  -\log(1-z_0) + \log(z_2) - \log(z_2 -1)\\ \\
  \log(z_3)-\log(z_4)\\ \\
  \log(z_0)+\log(1-z_1)-\log(1-z_2)-\log(1-z_3)-\log(1-z_4)
    \end{matrix} \right),
\end{equation*}
}
and
$$
\diff H_{\l}(z) = \left(
\begin{matrix}
	\frac{1}{1-z_0} & 0 & \frac{-1}{z_2(1-z_2)} & 0 & 0\\
	0 & 0 & 0 & \frac{1}{z_3} & \frac{-1}{z_4}\\
	\frac{1}{z_0} & \frac{-1}{1-z_1} & \frac{1}{1-z_2} & \frac{1}{1-z_3} & \frac{1}{1-z_4}
\end{matrix}
\right).
$$

Together with the differential of the log-curvature map we obtain:
$$
\diff ( G,H)(z) = \left(
	\begin{matrix}
	0 & \frac{-1}{z_1-1} & 0 & 0 & 0\\
	\frac{1}{z_0(z_0-1)} & \frac{-1}{z_1-1} & \frac{1}{z_2} & \frac{1}{z_3(z_3-1)} & \frac{1}{z_4(z_4-1)}\\
	0 & 0 & 0 & 0 & 0\\
	\frac{-1}{z_0(z_0-1)} & \frac{2}{z_1-1} & \frac{-1}{z_2} & \frac{-1}{z_3(z_3-1)} & \frac{-1}{z_4(z_4-1)}\\
	\frac{1}{1-z_0} & 0 & \frac{-1}{z_2(1-z_2)} & 0 & 0\\
	0 & 0 & 0 & \frac{1}{z_3} & \frac{-1}{z_4}\\
	\frac{1}{z_0} & \frac{-1}{1-z_1} & \frac{1}{1-z_2} & \frac{1}{1-z_3} & \frac{1}{1-z_4}
	\end{matrix} \right).
$$

The determinant of the minor obtained by removing the second and third row from $\diff ( G,H)(z)$ is
$$
-\frac{2(-1+z_0+z_3+z_4 + z_0 z_2 -z_0z_2z_3 -z_0z_2z_4 - z_3 z_4 -z_0z_3z_4+z_0z_2z_3z_4)}{z_0 z_1 z_2 z_3 z_4 ( 1- z_0)( 1- z_1)( 1- z_2)( 1- z_3)( 1- z_4)}$$

For all $u \in \ima(G)$, the restriction map
$$
\restr{H_{\l}}{G^{-1}(u)}: G^{-1}(u) \longrightarrow \C^{3}
$$
is a local diffeomorphism onto its image. In particular we can parameterize $G^{-1}(u)$ through $H_{\l}$. For $u^0$ as above, we get the following local parameterization around $z^0$,
$$
\begin{cases}
z_0 = 1 - z_4 k;\\
z_1 = e^{\frac{\pi}{3}}\\
z_2 = \frac{- z_4 k'}{1-z_4 k'}\\
z_3 = z_4 e^{t_2}\\
z_4 \mbox{ is the solution of the equation } \frac{k k'}{1-z_4 k'} = \frac{(1-z_4 k)e^{t_2+\frac{4}{3} \pi i}}{(z_4 e^{t_2}-1)(z_4-1)}
\end{cases}
$$
where $k = e^{\frac{5}{6}\pi i + \frac{t_2+t_3-t_1}{2}}$, $k' = e^{\frac{5}{6}\pi i + \frac{t_2+t_3+t_1}{2}}$ and for all $(t_1,t_2,t_3) \in \C^3$ close to $$H_{\l}(e^{\frac{\pi}{3}i},e^{\frac{\pi}{3}i},e^{\frac{\pi}{3}i},e^{\frac{\pi}{3}i},e^{\frac{\pi}{3}i})=(0,0,\pi i).$$
	

\textbf{Acknowledgements.}
The first author is supported by a Commonwealth of Australia International Postgraduate Research Scholarship.
The research of the second author is partially supported by the United States National Science Foundation grants NSF DMS 1222663, 1207832 and 1405106.
Research of the third author is partially supported by Australian Research Council grant DP140100158.
The authors are grateful to the referees for comments that improved this paper.


\bibliographystyle{amsplain}


\address{Alex Casella,\\ School of Mathematics and Statistics F07,\\ The University of Sydney,\\ NSW 2006 Australia\\
(casella@maths.usyd.edu.au)\\--}

\address{Feng Luo,\\ Department of Mathematics,\\ Rutgers University,\\ New Brunswick, NJ 08854, USA\\
(fluo@math.rutgers.edu)\\--}

\address{Stephan Tillmann,\\ School of Mathematics and Statistics F07,\\ The University of Sydney,\\ NSW 2006 Australia\\
(tillmann@maths.usyd.edu.au)}

\Addresses

\end{document}